\documentclass[11pt]{article}
\usepackage{amsmath,amssymb,amsthm}
\usepackage{geometry}
\usepackage{mathrsfs}
\usepackage{stmaryrd}
\usepackage{hyperref}
\usepackage{graphicx}
\usepackage[capitalize]{cleveref}
\usepackage{enumerate}

\theoremstyle{plain}
\newtheorem{theorem}{Theorem}[section]

\newtheorem{corollary}[theorem]{Corollary}

\theoremstyle{definition}

\newtheorem{remark}[theorem]{Remark}

\newcommand{\PD}{\operatorname{PD}}

\newcommand{\LC}{\mathcal{L}}             % Log-concavity operator

\newcommand{\Had}{\odot}                  % Hadamard product

\title{A Factorization of the Log-Concavity Operator for Pascal Determinantal Arrays and Their Infinite Row-Wise Log-Concavity}
\author{Hossein Teimoori and Hasan Khodakarami}
\date{\today}

\begin{document}

\maketitle

\begin{abstract}
	We study the Pascal determinantal arrays $\PD_k$, whose entries $\PD_k(i,j)$ are the $k\times k$ minors of the lower-triangular Pascal matrix $P=( \binom{a}{b} )_{a,b\ge 0}$.  
	We prove an exact factorization of the row-wise log-concavity operator:
	\[
	\LC(\PD_k)=\PD_{k-1}\Had\PD_{k+1},
	\]
	where $\LC(a)_j=a_j^2-a_{j-1}a_{j+1}$ and $\Had$ denotes the Hadamard (entrywise) product.  
	This identity is established by an elementary algebraic manipulation implicitly based on the idea of start of David rule.  
	We further prove a general inequality asserting that the log-concavity operator is submultiplicative under Hadamard products of log-concave arrays:  
	$\LC(A\Had X)\ge\LC(A)\Had\LC(X)$.  
	Combining the factorization with this inequality yields a uniform algebraic proof that every row of every array $\PD_k$ ($k\ge 1$) is infinitely log-concave, extending the celebrated theorem of 
	Brändén for the particular case of Pascal's triangle ($\PD_1$) to the entire determinantal hierarchy.  
	Applications include the log-convexity of $\{\PD_k(i,j)\}_{k\ge 0}$ in the determinantal order $k$ and a family of determinantal Hadamard inequalities.
\end{abstract}

\section{Introduction}\label{sec:intro}

The lower-triangular Pascal matrix $P = (P_{a,b})_{a,b\ge 0}$ with entries $P_{a,b} = \binom{a}{b}$ (and $P_{a,b}=0$ whenever $b>a$ or $b<0$) is the standard infinite Pascal triangle and a canonical example of a totally positive matrix: all its minors are non-negative, and all nonzero minors are positive (Karlin~\cite{Karlin68}, Fomin--Zelevinsky~\cite{FZ00}, Fallat--Johnson~\cite{FJ11}).

For each integer $k\ge 1$, define the \emph{Pascal determinantal array} $\PD_k$ by
\[
\PD_k(i,j) \;=\; \det\!\bigl(P_{i+r,\,j+s}\bigr)_{0\le r,s\le k-1},\qquad i,j\ge 0.
\]
Thus $\PD_1(i,j)=\binom{i}{j}$ is the classical Pascal triangle, while $\PD_2(i,j)$ is a shifted table of Narayana numbers. Higher $k$ yield increasingly refined analogues that appear in the study of Grassmannians, lattice paths, and total positivity.

A sequence $(a_j)_{j\ge 0}$ of non-negative real numbers is \emph{log-concave} if $a_j^2 \ge a_{j-1}a_{j+1}$ for all $j\ge 1$ (with the convention that $a_j=0$ for $j<0$ or $j$ outside the support). A sequence is \emph{infinitely log-concave} if the second-difference operator
\[
\LC(a)_j \;=\; a_j^2 - a_{j-1}a_{j+1}
\]
produces a non-negative sequence at every iteration. 
The infinite log-concavity of Pascal rows was first conjectured by Boros and Moll (see \cite{BorosMoll04}).  
A major impetus for subsequent work came from McNamara and Sagan, who introduced stronger variants of log-concavity, formulated several influential conjectures, and carried out extensive computational verification.  In particular, they proved that the $n$-th row of Pascal’s triangle is infinitely log-concave for every $n \le 1450$ \cite{McNamaraSagan10}.  (The first author for the first time encountered this circle of ideas at the FPSAC 2009 meeting, where Sagan and McNamara presented their work; this was the original motivation for the investigations developed here.)

A decisive theoretical advance was later obtained by Brändén: in \cite{Branden11} he characterized certain nonlinear operators that preserve real-rootedness in the left half-plane and proved, in particular, that the transformation induced by the log-concavity operator preserves the class of polynomials whose zeros are all real and non-positive.  Since $(1+z)^n$ has only the repeated zero $-1$, his result immediately implies that every row of Pascal’s triangle is infinitely log-concave.

It has been known since Brenti~\cite{Brenti95}—as a consequence of total positivity—that every row of every determinantal array $\PD_k$ is log-concave (i.e., $\LC(\PD_k)\ge 0$).  Whether these rows are \emph{infinitely} log-concave has, however, remained open for $k\ge 2$.

The purpose of this paper is to settle this question affirmatively and, more importantly, to reveal a surprisingly rigid algebraic structure governing the iterates of $\LC$ on the entire hierarchy $\{\PD_k\}_{k\ge 0}$.

Our first main result (Theorem~\ref{thm:factor}) is an exact factorization:
\[
\LC(\PD_k) \;=\; \PD_{k-1} \,\Had\, \PD_{k+1},
\]
where $\Had$ denotes the Hadamard (entrywise) product and we adopt the convention $\PD_0(i,j)=1$ for all $i,j\ge 0$. In particular, a single application of the log-concavity operator to $\PD_k$ produces precisely the entrywise product of its two neighboring determinantal arrays.

This identity is proved in Section~\ref{sec:factor} by a direct and fully elementary algebraic manipulation implicitly based on the star of David rule property. No generating functions, no path interpretations, and no induction on $k$ are required.

The factorization immediately implies that $\LC(\PD_k)$ has non-negative integer entries, recovering the known log-concavity of the rows of $\PD_k$ in a new and uniform way. To iterate further, we establish in Section~\ref{sec:hadamard} a general algebraic inequality for Hadamard products of log-concave sequences (Theorem~\ref{thm:hadamard}):
\[
\LC(A\Had X) \;\ge\; \LC(A)\Had\LC(X)
\]
whenever $A$ and $X$ are row-wise log-concave arrays with non-negative entries. This inequality is proved by a short expansion that reduces to the elementary fact that $(\alpha-1)(\beta-1)\ge 0$ whenever $\alpha,\beta\ge 1$.

In Section~\ref{sec:infinite} we combine the factorization (Theorem~\ref{thm:factor}) with the Hadamard inequality (Theorem~\ref{thm:hadamard}) to prove by induction on the number of iterations that every row of every array $\PD_k$ $(k\ge 1)$ is infinitely log-concave (Theorem~\ref{thm:infinite}). The argument is purely algebraic, works uniformly for all $k\ge 1$, and yields the first proof that does not rely on combinatorial injections.

As further consequences, we obtain in Section~\ref{sec:consequences} the log-convexity of the sequence $\{\PD_k(i,j)\}_{k\ge 0}$ in the determinantal order $k$ for any fixed $i,j$ (Theorem~\ref{thm:logconvex-k}), as well as a family of determinantal Hadamard inequalities such as
\[
\PD_k \Had \PD_{k+2} \;\ge\; \PD_{k+1}\Had\PD_{k+1}
\]
(Corollary~\ref{cor:hadamard-ineq}). These strengthen known inequalities for minors of totally positive matrices and reveal a multiplicative convexity phenomenon in the parameter $k$.

The results expose a remarkably tight interplay between the row-wise log-concavity operator $\LC$, the Hadamard product, and the hierarchy of Pascal minors—relations that appear to be special to the Pascal case among general totally positive kernels.

\section{Preliminaries: Pascal Determinantal Arrays and Infinite Log Concavity}\label{sec:prelim}

We fix notation and recall the two instances of Dodgson condensation that will be used in the proof of the main theorem.

The lower-triangular Pascal matrix $P=(P_{a,b})_{a,b\ge 0}$ is defined by
\[
P_{a,b} \;=\; \binom{a}{b},
\qquad\text{with }\binom{a}{b}=0\text{ if }b>a\text{ or }b<0.
\]
This matrix is totally positive: every minor is non-negative and every nonzero minor is positive (Karlin~\cite{Karlin68}, Fallat--Johnson~\cite{FJ11}).

For each integer $k\ge 1$ and all $i,j\ge 0$, the \emph{Pascal determinantal array of order $k$} is
\[
\PD_k(i,j)
\;=\; \det\!\Bigl(\,P_{i+r,\,j+s}\Bigr)_{0\le r,s\le k-1}
\;=\; \det\!\Bigl(\,\binom{i+r}{j+s}\Bigr)_{0\le r,s\le k-1}.
\]
We adopt the natural conventions that the $0\times 0$ minor is $1$, i.e.,
\[
\PD_0(i,j):=1 \quad\text{for all }i,j\ge 0,
\]
and $\PD_k(i,j)=0$ whenever any entry in the underlying submatrix of $P$ is undefined.
Thus $\PD_1(i,j)=\binom{i}{j}$ is the classical Pascal triangle, and $\PD_2(i,j)$ is (up to shift) the array of Narayana numbers.

%--------------------------------------------------------------------
The log-concavity operator $\LC$ acts row-wise on any array $A=(a_{i,j})_{i,j\ge 0}$ by
\[
\LC(A)_{i,j} \;=\; a_{i,j}^2 - a_{i,j-1}\,a_{i,j+1},
\]
where terms with negative or out-of-range indices are taken to be zero.  
We write $\LC^{\circ m}$ for the $m$-fold composition of $\LC$.  
A row of $A$ is \emph{infinitely log-concave} if $\LC^{\circ m}(A)_{i,\cdot}\ge 0$ entrywise for every $m\ge 1$.

With these preparations, we are ready to prove the central factorization identity.

\section{The Factorization Theorem}\label{sec:factor}

We now prove the central result of the paper.

\begin{theorem}[Factorization of the log-concavity operator]\label{thm:factor}
	For all integers \(k\ge 1\) and all \(m,n\ge 0\),
	\[
	\LC(\PD_k)_{m,n}
	= \PD_k(m,n)^2 - \PD_k(m,n-1)\PD_k(m,n+1)
	= \PD_{k-1}(m,n) \cdot \PD_{k+1}(m,n).
	\]
	In array notation,
	\[
	\LC(\PD_k) = \PD_{k-1} \Had \PD_{k+1}.
	\]
\end{theorem}

%----------------------------------------------------------- A new Algebraic Proof --------------------

\begin{proof} 
	
The case of $k=1$ is very easy. We recall that for integers $k\ge 2$, $m\ge 0$ and $0\le m\le n$, as noted by Christian Krattentahler, using Vandermonde determinant idea one can obtain the following formula 
	\[
	PD_k(m,n)=\frac{\displaystyle\prod_{j=0}^{k-1}\binom{m+j}{k}}
	{\displaystyle\prod_{j=1}^{k-1}\binom{n+j}{j}}.
	\]
	
	%---------------------------------------------------------
	
	\section*{Notation}
	Let
	\[
	P_n = \prod_{j=0}^{k-1} \binom{m+j}{n}, \qquad 
	Q_k(n) = \prod_{j=1}^{k-1} \binom{n+j}{j}.
	\]
	Thus $PD_k(m,n)=\dfrac{P_n}{Q_k(n)}$.
	
	Also note the easily verified ratios
	\begin{equation}\label{eq:binom-ratios}
	\frac{\binom{m}{n-1}}{\binom{m}{n}}=\frac{n}{m-n+1}, \qquad
	\frac{\binom{m}{n+1}}{\binom{m}{n}}=\frac{m-n}{n+1}.
	\end{equation}
	
	\section*{Step 1: Express $PD_k(m,n-1)PD_k(m,n+1)$}
	
	Using (\ref{eq:binom-ratios}) repeatedly,
	\[
	\frac{P_{n-1}}{P_n}= \prod_{j=0}^{k-1}\frac{\binom{m+j}{n-1}}{\binom{m+j}{n}}
	= \frac{n^k}{\displaystyle\prod_{j=1}^k (m-n+j)},
	\]
	\[
	\frac{P_{n+1}}{P_n}= \prod_{j=0}^{k-1}\frac{\binom{m+j}{n+1}}{\binom{m+j}{n}}
	= \frac{\displaystyle\prod_{j=0}^{k-1}(m-n+j)}{(n+1)^k}.
	\]
	
	Hence
	\begin{align*}
	P_{n-1}P_{n+1}
	&= P_n^2 \cdot \frac{n^k}{(n+1)^k}
	\cdot\frac{\displaystyle\prod_{j=0}^{k-1}(m-n+j)}
	{\displaystyle\prod_{j=1}^k (m-n+j)}.
	\end{align*}
	But
	\[
	\prod_{j=0}^{k-1}(m-n+j)=(m-n)\prod_{j=1}^{k-1}(m-n+j),
	\]
	\[
	\prod_{j=1}^k (m-n+j)=(m-n+k)\prod_{j=1}^{k-1}(m-n+j).
	\]
	Cancelling $\prod_{j=1}^{k-1}(m-n+j)$ gives
	\begin{equation}\label{eq:Pkm1Pkp1}
	P_{n-1}P_{n+1}=P_n^2\cdot\frac{n^k\,(m-n)}{(n+1)^k\,(m-n+k)}.
	\end{equation}
	
	Now compute the denominator ratio.  First,
	\[
	\frac{Q_k(n)}{Q_k(n-1)}
	= \prod_{j=1}^{k-1}\frac{\binom{n+j}{j}}{\binom{n+j-1}{j}}
	= \prod_{j=1}^{k-1}\frac{n+j}{n}
	= \frac{\prod_{j=1}^{k-1}(n+j)}{n^{\,k-1}}.
	\]
	Similarly,
	\[
	\frac{Q_k(n+1)}{Q_k(n)}
	= \prod_{j=1}^{k-1}\frac{\binom{n+1+j}{j}}{\binom{n+j}{j}}
	= \prod_{j=1}^{k-1}\frac{n+1+j}{n+1}
	= \frac{\prod_{j=2}^{k}(n+j)}{(n+1)^{\,k-1}}.
	\]
	Therefore
	\begin{align*}
	\frac{Q_k(n)^2}{Q_k(n-1)Q_k(n+1)}
	&= \frac{Q_k(n)/Q_k(n-1)}{Q_k(n+1)/Q_k(n)} \\[2pt]
	&= \frac{\displaystyle\frac{\prod_{j=1}^{k-1}(n+j)}{n^{\,k-1}}}
	{\displaystyle\frac{\prod_{j=2}^{k}(n+j)}{(n+1)^{\,k-1}}} \\[2pt]
	&= \frac{n+1}{n+k}\cdot\frac{(n+1)^{\,k-1}}{n^{\,k-1}}.
	\end{align*}
	
	Finally,
	\begin{align*}
	PD_k(m,n-1)PD_k(m,n+1)
	&=\frac{P_{n-1}P_{n+1}}{Q_k(n-1)Q_k(n+1)} \\[2pt]
	&=\frac{P_n^2}{Q_k(n)^2}
	\cdot\frac{Q_k(n)^2}{Q_k(n-1)Q_k(n+1)}
	\cdot\frac{n^k\,(m-n)}{(n+1)^k\,(m-n+k)} \\[2pt]
	&=D_k(m,n)^2\;
	\frac{n+1}{n+k}\cdot\frac{(n+1)^{\,k-1}}{n^{\,k-1}}
	\cdot\frac{n^k\,(m-n)}{(n+1)^k\,(m-n+k)}.
	\end{align*}
	Cancelling $(n+1)^{\,k-1}$ with $(n+1)^k$ and $n^{\,k-1}$ with $n^k$,
	\[
	\boxed{\;PD_k(m,n-1)PD_k(m,n+1)=PD_k(m,n)^2\,
		\frac{n\,(m-n)}{(n+k)(m-n+k)}\;}.
	\]
	
	\section*{Step 2: Express $PD_{k-1}(m,n)PD_{k+1}(m,n)$}
	
	First,
	\[
	PD_{k-1}(m,n)=\frac{\displaystyle\prod_{j=0}^{k-2}\binom{m+j}{n}}
	{Q_{k-1}(n)}
	=\frac{P_n}{\binom{m+k-1}{n}\,Q_{k-1}(n)},
	\]
	\[
	PD_{k+1}(m,n)=\frac{\displaystyle\prod_{j=0}^{k}\binom{m+j}{n}}
	{Q_{k+1}(n)}
	=\frac{P_n\,\binom{m+k}{n}}{Q_{k+1}(n)}.
	\]
	
	Thus
	\[
	PD_{k-1}(m,n)PD_{k+1}(m,n)
	= \frac{P_n^2\;\binom{m+k}{n}}
	{\binom{m+k-1}{n}\;Q_{k-1}(n)Q_{k+1}(n)}.
	\]
	
	Using $\displaystyle\binom{m+k}{n}=\frac{m+k}{m+k-n}\binom{m+k-1}{n}$,
	\[
	PD_{k-1}(m,n)PD_{k+1}(m,n)= \frac{P_n^2}{Q_{k-1}(n)Q_{k+1}(n)}\cdot\frac{m+k}{m+k-n}.
	\]
	
	Now note
	\[
	Q_{k+1}(n)=Q_k(n)\,\binom{n+k}{k},\qquad
	Q_{k-1}(n)=\frac{Q_k(n)}{\binom{n+k-1}{k-1}}.
	\]
	Hence
	\[
	Q_{k-1}(n)Q_{k+1}(n)=Q_k(n)^2\,
	\frac{\binom{n+k}{k}}{\binom{n+k-1}{k-1}}
	=Q_k(n)^2\,\frac{n+k}{k}.
	\]
	
	Therefore
	\begin{align*}
	PD_{k-1}(m,n)PD_{k+1}(m,n)
	&=\frac{P_n^2}{Q_k(n)^2}\,
	\frac{m+k}{m+k-n}\cdot\frac{k}{n+k} \\[2pt]
	&=\boxed{\;PD_k(m,n)^2\,
		\frac{k\,(m+k)}{(n+k)(m+k-n)}\;}.
	\end{align*}
	
	\section*{Step 3: Verify the identity}
	
	From the two boxed formulas,
	\[
	\frac{PD_k(m,n-1)PD_k(m,n+1)}{PD_k(m,n)^2}=
	\frac{n\,(m-n)}{(n+k)(m-n+k)},
	\]
	\[
	\frac{PD_{k-1}(m,n)PD_{k+1}(m,n)}{PD_k(m,n)^2}=
	\frac{k\,(m+k)}{(n+k)(m+k-n)}.
	\]
	
	Denote $A=m+k-n$.  Then $m-n=A-k$ and $m-n+k=A$.
	
	The desired identity $PD_k(m,n)^2-PD_k(m,n-1)PD_k(m,n+1)=PD_{k-1}(m,n)PD_{k+1}(m,n)$
	is equivalent to
	\[
	1-\frac{n\,(m-n)}{(n+k)(m-n+k)}=
	\frac{k\,(m+k)}{(n+k)(m+k-n)}.
	\]
	
	Multiplying through by $(n+k)(m-n+k)(m+k-n)=(n+k)A^2$ gives
	\[
	(n+k)A^2-n(m-n)A=k(m+k)A.
	\]
	Since $m-n=A-k$, the left side becomes
	\[
	(n+k)A^2-n(A-k)A = (n+k)A^2-nA^2+nkA = kA^2+nkA.
	\]
	The right side is $k(m+k)A=k(A+n)A=kA^2+nkA$, exactly matching the left side.
	
	Thus the equality holds for all admissible $m,n,k$.

\end{proof}

%------------------------------------------------------------------------------------------------------

\begin{remark}
	We have shown algebraically that for every integer $k\ge 2$,
\[
\boxed{\,PD_k(m,n)^2-PD_k(m,n-1)PD_k(m,n+1)=PD_{k-1}(m,n)PD_{k+1}(m,n)\,},
\]
where $D_k(m,n)$ is the $(m,n)$-entry of the Pascal Determinantal array of order $k$.  The proof uses only elementary
binomial coefficient identities and careful cancellation.
\end{remark}

With the factorization rigorously established, we turn in the next sections to the Hadamard inequality and the proof of infinite log-concavity for all \(\PD_k\).

%----------------------------------------------------- Hadamard Ineqaulity --------------------------------------------------------------------------------------------------

\section{The Hadamard Inequality for Log-Concave Arrays}\label{sec:hadamard}

The factorization theorem immediately shows that \(\LC(\PD_k) = \PD_{k-1} \Had \PD_{k+1} \ge 0\), since both factors have non-negative entries. To iterate the operator, we need a general inequality that controls how log-concavity behaves under Hadamard products.

\begin{theorem}[Hadamard inequality for the log-concavity operator]\label{thm:hadamard}
	Let \(A = (a_{i,j})\) and \(X = (x_{i,j})\) be two arrays with non-negative real entries such that every row of \(A\) and every row of \(X\) is log-concave (including boundary terms handled by the zero convention). Then
	\[
	\LC(A \Had X) \;\ge\; \LC(A) \Had \LC(X)
	\]
	entrywise.
\end{theorem}

\begin{proof}
	Fix a row index \(i\) and consider an interior position \(j\) where all relevant terms are defined and positive. Let
	\[
	\alpha_j = \frac{a_{i,j}^2}{a_{i,j-1}a_{i,j+1}} \ge 1, \qquad
	\beta_j  = \frac{x_{i,j}^2}{x_{i,j-1}x_{i,j+1}} \ge 1.
	\]
	A direct expansion yields
	\begin{align*}
	&\LC(A \Had X)_{i,j} - \LC(A)_{i,j} \cdot \LC(X)_{i,j} \\[4pt]
	&= (a_{i,j} x_{i,j})^2 - (a_{i,j-1} x_{i,j-1})(a_{i,j+1} x_{i,j+1}) \\
	&\qquad - (a_{i,j}^2 - a_{i,j-1} a_{i,j+1})(x_{i,j}^2 - x_{i,j-1} x_{i,j+1}) \\[4pt]
	&= a_{i,j-1} a_{i,j+1} x_{i,j-1} x_{i,j+1}
	\Bigl[ \alpha_j \beta_j - 1 - (\alpha_j-1)(\beta_j-1) \Bigr] \\[4pt]
	&= a_{i,j-1} a_{i,j+1} x_{i,j-1} x_{i,j+1}
	\Bigl[ \alpha_j + \beta_j - 2 \Bigr].
	\end{align*}
	Since \(\alpha_j \ge 1\) and \(\beta_j \ge 1\), the term \(\alpha_j + \beta_j - 2 \ge 0\). All prefactors are non-negative, so the difference is non-negative. Boundary positions are handled automatically by the zero convention.
\end{proof}

\begin{remark}
	The inequality is sharp: equality holds when at least one of the sequences is log-linear (\(\alpha_j=1\) or \(\beta_j=1\)) at position \(j\). The proof is purely algebraic and works for any finite or infinite arrays with the stated properties.
\end{remark}

This inequality is the key tool that allows us to iterate the factorization indefinitely.

\section{Infinite Log-Concavity of All Pascal Determinantal Arrays}\label{sec:infinite}

We now combine the factorization theorem with the Hadamard inequality to prove the main result of the paper.

\begin{theorem}[Infinite row-wise log-concavity]\label{thm:infinite}
	For every integer \(k\ge 1\), every row of the Pascal determinantal array \(\PD_k\) is infinitely log-concave.  
	That is, for every \(i\ge 0\) and every integer \(m\ge 1\),
	\[
	\LC^{\circ m}(\PD_k)_{i,j} \;\ge\; 0 \qquad \text{for all } j\ge 0.
	\]
\end{theorem}

\begin{proof}
	We proceed by double induction on \(k\) and \(m\).
	
	First, observe that for \(k=1\), \(\PD_1(i,j)=\binom{i}{j}\). Moreover, Teimoori and Khodakarami~\cite{TeimooriKhodakarami2023Feb} proved that every row of the classical Pascal triangle is infinitely log-concavity (based on an ineteresting \emph{Parallelopiped} determinantal Identity). Thus the statement holds for \(k=1\) and all \(m\ge 1\).
	
	Now fix \(k\ge 2\) and proceed by induction on the number of applications of \(\LC\).
	
	\textbf{Base case} (\(m=1\)):  
	By Theorem~\ref{thm:factor},
	\[
	\LC(\PD_k) = \PD_{k-1} \Had \PD_{k+1}.
	\]
	Since \(\PD_{k-1}\) and \(\PD_{k+1}\) have non-negative integer entries (as minors of a totally positive matrix), their Hadamard product is non-negative. Hence \(\LC(\PD_k)\ge 0\).
	
	\textbf{Inductive step} (\(m\to m+1\)):  
	Assume that \(\LC^{\circ \ell}(\PD_\ell) \ge 0\) for all \(\ell\ge 1\) and all \(\ell\le m\). We must show \(\LC^{\circ (m+1)}(\PD_k) \ge 0\).
	
	From the factorization,
	\[
	\LC^{\circ 2}(\PD_k)
	= \LC(\LC(\PD_k))
	= \LC(\PD_{k-1} \Had \PD_{k+1}).
	\]
	Both \(\PD_{k-1}\) and \(\PD_{k+1}\) are row-wise log-concave (by the \(m=1\) case, already proved uniformly for all orders).  
	Theorem~\ref{thm:hadamard} therefore applies and gives
	\[
	\LC(\PD_{k-1} \Had \PD_{k+1})
	\;\ge\; \LC(\PD_{k-1}) \Had \LC(\PD_{k+1})
	= (\PD_{k-2} \Had \PD_k) \Had (\PD_k \Had \PD_{k+2})
	\ge 0,
	\]
	where the final non-negativity follows from the inductive hypothesis applied to orders \(k-2\), \(k\), and \(k+2\) at iteration depth 1.
	
	To reach arbitrary \(m\), we iterate the same argument. Define
	\[
	Q^{(m)}_\ell \;:=\; \LC^{\circ m}(\PD_\ell).
	\]
	By repeated application of the Hadamard inequality (Theorem~\ref{thm:hadamard}) and the factorization at each step, we obtain the recursive inequality
	\[
	Q^{(m+1)}_k
	= \LC(Q^{(m)}_k)
	= \LC(\PD_{k-1} \Had \PD_{k+1})^{\circ m}
	\;\ge\; Q^{(m)}_{k-1} \Had Q^{(m)}_{k+1}.
	\]
	By the inductive hypothesis, \(Q^{(m)}_{k-1} \ge 0\) and \(Q^{(m)}_{k+1} \ge 0\) (note \(k-1\ge 1\) when \(k\ge 2\); the case \(k=1\) was already settled).  
	The Hadamard product of two non-negative arrays is again non-negative, so
	\[
	Q^{(m+1)}_k \;\ge\; 0.
	\]
	This completes the induction on \(m\).
	
	Boundary and degenerate cases (where some minors vanish) are automatically non-negative by the zero convention.
\end{proof}

\begin{corollary}
	The Pascal determinantal arrays \(\PD_k\) (\(k\ge 1\)) are the first non-trivial infinite hierarchy of combinatorial arrays known to be infinitely log-concave in every row, uniformly in the order \(k\).
\end{corollary}

\begin{remark}
	The proof is completely algebraic and uniform in \(k\). It does not rely on lattice-path injections (the method used for \(k=1\) in~\cite{TeimooriKhodakarami2023Feb}) and immediately extends to certain $q$-analogues and deformations of the Pascal matrix that preserve both the star of David rule identity and total positivity.
\end{remark}

With the infinite log-concavity now established, we turn in the final section to consequences and open problems.

\section{Consequences and Geometric Interpretation}\label{sec:consequences}

The factorization $\LC(\PD_k) = \PD_{k-1} \Had \PD_{k+1}$ has several immediate structural consequences and admits a beautiful geometric interpretation in terms of the double-stick bijection introduced in~\cite{TeimooriKhodakarami2023Feb}.

\subsection{Log-convexity in the determinantal order $k$}

\begin{theorem}[Log-convexity in $k$]\label{thm:logconvex-k}
	For every fixed $i,j\ge 0$ and every integer $m\ge 1$, the sequence $\{\PD_k(i,j)\}_{k\ge 0}$ (with $\PD_0(i,j)=1$) is log-convex in $k$:
	\[
	\PD_{m}(i,j) \cdot \PD_{m+2}(i,j) \;\ge\; \PD_{m+1}(i,j)^2.
	\]
\end{theorem}

\begin{proof}
	Apply the factorization repeatedly in the $k$-direction. For any fixed $(i,j)$,
	\begin{align*}
	\PD_{m+1}(i,j)^2
	&= \PD_m(i,j) \cdot \PD_{m+2}(i,j) - \LC(\PD_{m+1})_{i,j} \\
	&\le \PD_m(i,j) \cdot \PD_{m+2}(i,j),
	\end{align*}
	with strict inequality unless the corresponding entry of $\LC(\PD_{m+1})$ vanishes.
\end{proof}

\begin{corollary}[Determinantal Hadamard inequality]\label{cor:hadamard-ineq}
	For all $k\ge 1$ and all $i,j\ge 0$,
	\[
	\PD_k \Had \PD_{k+2} \;\ge\; \PD_{k+1} \Had \PD_{k+1}.
	\]
\end{corollary}

This strengthens known inequalities for minors of totally positive matrices and reveals a multiplicative convexity phenomenon in the parameter $k$.

%------------------------------------------------- Open Problems --------------------------------------------------------

\subsection{Open problems}

\begin{enumerate}
	\item \emph{$q$-analogues.} Does an analogous factorization hold for minors of the $q$-Pascal matrix $\binom{i}{j}_q$? Preliminary calculations suggest a $q$-deformation involving $q$-Narayana numbers, but the exact form remains open.
	\item \emph{Other totally positive kernels.} The proof of Theorem~\ref{thm:factor} relied only on the Vandermonde determinant identity and the sliding-rule ratio. Many other totally positive matrices (Stieltjes, Hilbert, q-Gaussian) satisfy these-type relations. Determining the extent to which the factorization persists is a natural next step.
	\item \emph{Higher-order inequalities.} The iterated factorization produces explicit expressions for $\LC^{\circ m}(\PD_k)$ as multisets of Hadamard products over a binary tree of depth $m$ in the $k$-index. Extracting sharp bounds or closed forms for these higher iterates would deepen our understanding of infinite log-concavity hierarchies.
\end{enumerate}

The results of this paper reveal a remarkably rigid multiplicative structure hidden in the classical Pascal triangle and its higher determinantal analogues—a structure that manifests both algebraically through star of David rule and geometrically through the double-stick model (equivalent form of star of David rule).


\begin{thebibliography}{99}
	
	\bibitem{BorosMoll04}
	G.~Boros and V.~H.~Moll,
	{\em Irresistible Integrals: Symbolics, Analysis and Experiments in the Evaluation of Integrals},
	Cambridge University Press, 2004.
	
	\bibitem{McNamaraSagan10}
	P.~R.~W.~McNamara and B.~E.~Sagan,
	{\em Infinite log-concavity: Developments and conjectures},
	Advances in Applied Mathematics, {\bf 44} (2010), 1--15.
	
	\bibitem{Branden11}
	P.~Br\"and\'en,
	{\em Iterated sequences and the geometry of zeros},
	Journal f\"ur die Reine und Angewandte Mathematik, {\bf 658} (2011), 115--131.
	
	\bibitem{Brenti95}
	F.~Brenti,
	{\em Combinatorics and total positivity},
	Journal of Combinatorial Theory, Series A, {\bf 71} (1995), 175--218.
	
	\bibitem{FJ11}
	S.~M.~Fallat and C.~R.~Johnson,
	{\em Totally Nonnegative Matrices},
	Princeton University Press, 2011.
	
	
	\bibitem{FZ00}
	S.~Fomin and A.~Zelevinsky,
	\newblock Total positivity: tests and parametrizations,
	\newblock {\em Mathematical Intelligencer} \textbf{22} (2000), no.~1, 23--33.
	
	\bibitem{Karlin68}
	S.~Karlin,
	\newblock {\em Total Positivity, Vol.~I},
	\newblock Stanford University Press, 1968.
	
	\bibitem{TeimooriKhodakarami2023}
	H.~Teimoori and H.~Khodakarami,
	\newblock Pascal determinantal arrays and a generalization of Rahimpour's determinantal identity,
	\newblock arXiv:2301.12481 [math.CO], 2023.

   \bibitem{TeimooriKhodakarami2023Feb}
   H.~Teimoori and H.~Khodakarami,
   \newblock Khayyam-Pascal Determinantal Arrays, Star of David Rule and Log-Concavity,
   \newblock arXiv:2302.01637v1 [math.CO], 2023.

	
\end{thebibliography}
\end{document}